\documentclass[times]{elsarticle}

\usepackage{amsmath}
\usepackage{amssymb}
\usepackage{amsfonts}
\usepackage{amsthm}
\usepackage{enumerate}

\swapnumbers
\newtheorem{theorem}{Theorem}[section]
\newtheorem{lemma}[theorem]{Lemma}
\newtheorem{corollary}[theorem]{Corollary}
\newdefinition{definition}[theorem]{Definition}
\newdefinition{remark}[theorem]{Remark}

\renewcommand{\a}{\mathfrak{a}}
\renewcommand{\b}{\mathfrak{b}}
\renewcommand{\c}{\mathfrak{c}}
\renewcommand{\d}{\mathfrak{d}}
\newcommand{\e}{\mathfrak{e}}

\renewcommand{\o}{\mathfrak{o}}
\newcommand{\p}{\mathfrak{p}}

\renewcommand{\v}{\mathrm{v}}

\newcommand{\GL}{\operatorname{GL}}
\newcommand{\N}{\mathbb{N}}
\newcommand{\divides}{\mathrel|}

\begin{document}
\title{Determinantal divisors of products of matrices over Dedekind domains}

\author[lama]{Marc Ensenbach}
\ead{Marc.Ensenbach@matha.rwth-aachen.de}
\address[lama]{RWTH Aachen, Lehrstuhl A f\"ur Mathematik, 52056 Aachen, Germany}

\begin{abstract}
Given three lists of ideals of a Dedekind domain, the question is raised whether there exist two matrices $A$ and $B$ with entries in the given Dedekind domain, such that the given lists of ideals are the determinantal divisors of $A$, $B$, and $A B$, respectively. To answer this question, necessary and sufficient conditions are developed in this article.
\end{abstract}


\maketitle

\section{Introduction}

When dealing with Hecke algebras with respect to a group $G$ and a subgroup $U$, one is interested in the characterisation of $k \in U g U$ and $k \in U g U h U$ for $g, h, k \in G$, since these conditions appear in formulas for the calculation of products in the algebra (\cite{Krieg:Hecke} I\,(4.4)). In the case that $U$ is the unimodular group $\GL_n(\o)$ over a Dedekind domain $\o$, the condition $k \in U g U$ can be characterised by the comparison of determinantal divisors (\cite{Krull:Matrizen} Satz 11). Considering the condition $k \in U g U h U$, one has to deal with determinantal divisors of products of matrices. Since the obtained results are of interest not only in the field of Hecke algebras (see e.\,g. \cite{Flanders:Elementary}), in this text they are developed and presented without any means and notation from Hecke theory; applications to Hecke theory will be published in another article named ``Hecke algebras related to unimodular groups over Dedekind domains''.

\section{Preliminaries and notation}

Denote by $\o$ a Dedekind domain, by $K$ its field of fractions, and by $\o^*$ its group of unities. For every integer $n$ let $I_n$ be the set of ($n \times n$) matrices with entries in $\o$ and non-zero determinant; furthermore, denote by $U_n$ the set of matrices in $I_n$ with determinant in $\o^*$ (in other words $U_n = \GL_n(\o)$ and $I_n = \GL_n(K) \cap \o^{n \times n}$). Following \cite{Steinitz:Rechteckige}, for every $k, n \in \N$ with $k \leq n$ fix an enumeration $(M_{n, k, i})_i$ of the subsets of $\{1, \ldots, n\}$ containing exactly $k$ elements, and for $1 \leq i, j \leq \smash{\binom{n}{k}}$ denote by $A_{(ij)}$ the submatrix of $A$ obtained from an $A \in \o^{n \times n}$ by deleting all but the rows numbered by elements of $M_{n, k, i}$ and all but the columns numbered by elements of $M_{n, k, j}$. Then the $k$-th derivative of $A$ is the matrix $A^{[k]} := (\tilde{a}_{ij})$ with $\tilde{a}_{ij} = \det A_{(ij)}$ for all $1 \leq i, j \leq \smash{\binom{n}{k}}$. With this notation, the $k$-th determinantal divisor of $A$ -- the g.\,c.\,d. (in an ideal theoretic sense) of all entries of $A^{[k]}$ -- is denoted by $\d_k(A)$. For convenience let $\d_k(A) = \o$ for all $k \leq 0$ and $\d_k(A) = \{0\}$ for all $k > n$. Then define the $k$-th elementary divisor $\e_k(A)$ by $\e_k(A) = \d_k(A) \d_{k - 1}(A)^{-1}$, if $\d_k(A) \neq \{0\}$, and $\e_k(A) = \{0\}$ otherwise. (From the general results in \cite{Franz:Elementarteilertheorie} one can conclude that all $\e_k(A)$ are ideals in $\o$.) By these definitions it is possible to exchange determinantal divisors for elementary divisors without loss of information. With this notation, the \emph{rank} of $A$ is biggest $r \in \N_0$ satisfying $\d_r(A) \neq \{0\}$, and the \emph{column class} $\mathfrak{C}(A)$ of $A$ is the ideal class of the g.\,c.\,d. of a nonzero column of $A^{[r]}$ (which is independent of the choice of the column).

In order to give handy formulations of the developed theorems, in this article the triple $((\a_1, \ldots, \a_n), (\b_1, \ldots, \b_n), (\c_1, \ldots, \c_n))$ is called \emph{realisable,} if there exist matrices $A, B \in I_n$ such that $\d_k(A) = \a_k$ and $\d_k(B) = \b_k$ as well as $\d_k(A B) = \c_k$ hold for all $1 \leq k \leq n$.

Before dealing with the existence of products in the next sections, it is sensible to characterise the existence of single matrices with prescribed determinantal divisors or elementary divisors first. One can cite the following slight variation of \cite{Franz:Elementarteilertheorie} Satz 7.

\begin{theorem}\label{texsingle}
Let $n \in \N$ with $n \geq 2$ and $\b_1, \ldots, \b_n$ be ideals in $\o$. Then there exists a matrix $A \in \o^{n \times n}$ satisfying $\e_k(A) = \b_k$ for every $1 \leq k \leq n$, if and only if $\b_k \divides \b_{k + 1}$ for all $1 \leq k < n$ and $\b_1 \cdots \b_n$ is a principal ideal (including $\{0\}$).
\end{theorem}

To end this section of preliminaries, two known auxilliary results that are needed in the rest of this article are stated. The first is part of \cite{Krull:Matrizen} Satz 11.

\begin{theorem}\label{tequaldc}
Let $n \in \N$ and $A, B \in \o^{n \times n}$. Then there exist $P, Q \in U_n$ satisfying $B = P A Q$ if and only if $\d_k(A) = \d_k(B)$ for all $1 \leq k \leq n$ and $\mathfrak{C}(A) = \mathfrak{C}(B)$. (In the case $A, B \in I_n$, the condition $\mathfrak{C}(A) = \mathfrak{C}(B)$ is always true.)
\end{theorem}

A slight variation of the normal form constructed in the considerations preceding \cite{Krull:Matrizen} Satz 11 directly yields the following theorem.

\begin{theorem}\label{tnormalform}
Let $A \in I_n$ and $A' = (\begin{smallmatrix} A & 0 \\ 0 & 0 \end{smallmatrix}) \in \o^{2 n \times 2 n}$. Then there exist $A_1, \ldots, A_n \in \o^{2 \times 2}$ and $P, Q \in U_{2 n}$, such that $A_k = (\begin{smallmatrix} * & 0 \\ * & 0 \end{smallmatrix})$ and $\d_1(A_k) = \e_k(A)$ for all $1 \leq k \leq n$ and
\[
P A' Q = \begin{pmatrix}
A_1 \\
& \ddots \\
& & A_n
\end{pmatrix}.
\]
\end{theorem}

\section{Necessary conditions for realisability}

In this section, some results on determinantal divisors of products of given matrices are shown. These results are then translated into propositions on realisability.

\begin{theorem}\label{texprodneclb}
For every $A, B \in I_n$ one has $\d_k(A) \d_k(B) \divides \d_k(A B)$ for all $1 \leq k \leq n$.
\end{theorem}
\begin{proof}
By definition $\d_k(A) = \d_1(A^{[k]})$ holds. Cauchy-Binets formula (\cite{Aitken:Determinants} 39) then yields $\d_k(A B) = \d_1((A B)^{[k]}) = \d_1(A^{[k]} B^{[k]})$, and since $\d_1(C) \d_1(D) \divides \d_1(C D)$ holds for all matrices in $\o^{m \times m}$, one obtains $\d_k(A) \d_k(B) = \d_1(A^{[k]}) \d_1(B^{[k]}) \divides \d_1(A^{[k]} B^{[k]}) = \d_k(A B)$.
\end{proof}

The just proven theorem can be seen as a ``lower bound'' for $\d_k(A B)$. An upper bound, which generalises an unpublished result of Koecher (\cite{Koecher:Matrices} Thm. I.7.1), can also be given.

\begin{theorem}\label{texprodnecub}
For every $A, B \in I_n$ and $1 \leq k \leq n$ one has
\[
\d_k(A B) \divides \d_k(A) \d_n(B) \d_{n - k}(B)^{-1} + \d_k(B) \d_n(A) \d_{n - k}(A)^{-1}.
\]
\end{theorem}
\begin{proof}
First assume that $\o$ is a principal ideal domain. According to the Smith normal form (\cite{Newman:Integral} Thm. II.9) there exist $a_1, \ldots, a_n, b_1, \ldots, b_n \in \o$ and $P, Q_1, Q_2, R \in U_n$ such that
\[
A' = \begin{pmatrix}
a_1 \\
& \ddots \\
& & a_n
\end{pmatrix} \qquad\text{and}\qquad B' = \begin{pmatrix}
b_1 \\
& \ddots \\
& & b_n
\end{pmatrix}
\]
satisfy $A = P A' Q_1$ and $B = Q_2 B' R$. Then
\[
\d_k(A B) = \d_k(P A' Q_1 Q_2 B' R) = \d_k(A' Q_1 Q_2 B')
\]
since determinantal divisors of a matrix are invariant under multiplication with elements of $U_2$ (\cite{Steinitz:Rechteckige} 10). Let $1 \leq r \leq \smash{\binom{n}{k}}$ such that $M_{n, k, r} = \{1, \ldots, k\}$. By the definition of determinantal divisors,
\[
\d_k(A B) = \d_k(A' Q_1 Q_2 B') \divides \det (A' Q_1 Q_2 B')_{(r s)}
\]
holds for all $1 \leq s \leq \smash{\binom{n}{k}}$, and since
\[
\det (A' Q_1 Q_2 B')_{(r s)} = a_1 \cdots a_k \cdot \det (Q_1 Q_2 B')_{(r s)} \divides a_1 \cdots a_n \cdot \det (Q_1 Q_2)_{(r s)} \cdot b_{n - k + 1} \cdots b_n
\]
is true for all $1 \leq s \leq \smash{\binom{n}{k}}$, one obtains
\[
\d_k(A B) \divides a_1 \cdots a_k \cdot b_{n - k + 1} \cdots b_n \cdot \sum_{1 \leq s \leq \binom{n}{k}} \det (Q_1 Q_2)_{(r s)} \o = a_1 \cdots a_k \cdot b_{n - k + 1} \cdots b_n
\]
since the g.\,c.\,d. of the $\det (Q_1 Q_2)_{(r s)} \o$ equals $\o$ (otherwise Leibniz' determinant formula would yield $\det (Q_1 Q_2) \notin \o^*$, which contradicts $Q_1 Q_2 \in U_n$). This shows the relation $\d_k(A B) \divides \d_k(A) \d_n(B) \d_{n - k}(B)^{-1}$, and $\d_k(A B) \divides \d_k(B) \d_n(A) \d_{n - k}(A)^{-1}$ can be proven analogously, which completes the proof of the proposition for the case that $\o$ is a principal ideal domain.

Now consider an arbitrary Dedekind domain $\o$ and the localisation $\o_\p$ of $\o$ at a prime ideal $\p$ of $\o$. Denote by $\a_1, \ldots, \a_n, \b_1, \ldots, \b_n, \c_1, \ldots, \c_n$ the determinantal divisors of $A$ and $B$ as well as $A B$, where the matrices are considered to be elements of $\o_\p^{n \times n}$. Since $\o_\p$ is a principal ideal domain, the already proven part yields $\c_k \divides \a_k \b_n \smash{\b_{n - k}^{-1}} + \b_k \a_n \smash{\a_{n - k}^{-1}}$ and thus $\a_{n - k} \b_{n - k} \c_k \divides \a_{n - k} \a_k \b_n + \b_{n - k} \b_k \a_n$. Using
\[
\v_\p(\a_{n - k} \b_{n - k} \c_k \cap \o) = \v_\p(\a_{n - k} \cap \o) + \v_\p(\b_{n - k} \cap \o) + \v_\p(\c_k \cap \o) = \v_\p(\d_{n - k}(A) \d_{n - k}(B) \d_k(A B))
\]
and the similarly obtained relation
\[
\v_\p((\a_{n - k} \a_k \b_n + \b_{n - k} \b_k \a_n) \cap \o) = \v_\p(\d_{n - k}(A) \d_k(A) \d_n(B) + \d_{n - k}(B) \d_k(B) \d_n(A)),
\]
from $\a_{n - k} \b_{n - k} \c_k \divides \a_{n - k} \a_k \b_n + \b_{n - k} \b_k \a_n$ follows
\[
\v_\p(\d_{n - k}(A) \d_{n - k}(B) \d_k(A B)) \leq \v_\p(\d_{n - k}(A) \d_k(A) \d_n(B) + \d_{n - k}(B) \d_k(B) \d_n(A)),
\]
and since this is valid for every prime ideal $\p$ of $\o$, this yields that $\d_{n - k}(A) \d_{n - k}(B) \d_k(A B)$ divides $\d_{n - k}(A) \d_k(A) \d_n(B) + \d_{n - k}(B) \d_k(B) \d_n(A)$ and thus proves the proposition.
\end{proof}

Now the achieved results are put together.

\begin{corollary}\label{cexprodnec}
The triple $((\a_1, \ldots, \a_n), (\b_1, \ldots, \b_n), (\c_1, \ldots, \c_n))$ is not realisable, if one of the following conditions is violated:
\begin{enumerate}[(1)]
\item $\a_n$, $\b_n$, and $\c_n$ are principal ideals.
\item $\a_1^2 \divides \a_2$ and $\b_1^2 \divides \b_2$ as well as $\c_1^2 \divides \c_2$ hold.
\item $\a_{k-1}^2 \a_{k-2}^{-1} \divides \a_k$ and $\b_{k-1}^2 \b_{k-2}^{-1} \divides \b_k$ as well as $\c_{k-1}^2 \c_{k-2}^{-1} \divides \c_k$ hold for all $3 \leq k \leq n$.
\item $\a_n \b_n = \c_n$.
\item $\a_k \b_k \divides \c_k$ holds for all $1 \leq k < n$.
\item $\c_k \divides \a_k \b_k (\a_n \smash{\a_k^{-1}} \smash{\a_{n - k}^{-1}} + \b_n \smash{\b_k^{-1}} \smash{\b_{n - k}^{-1}})$ holds for all $1 \leq k < n$.
\end{enumerate}
\end{corollary}
\begin{proof}
(1)--(3) come from \ref{texsingle}, (4) is the multiplicativity of the determinant, (5) comes from \ref{texprodneclb}, and (6) from \ref{texprodnecub}
\end{proof}

\section{Sufficient conditions for realisability}

In this section, it is shown that in certain circumstances the realisability of determinantal divisors of a product can be assured. Therefore first an auxiliary result has to be stated.

\begin{lemma}\label{ldiag}
Let $n \in \N$ and $A_1, \ldots, A_n \in \o^{2 \times 2}$ be matrices of rank $1$ satisfying the divisablity condition $\d_1(A_1) \divides \d_1(A_2) \divides \cdots \divides \d_1(A_n)$. Let
\[
A = \begin{pmatrix}
A_1 \\
& \ddots \\
& & A_n
\end{pmatrix}.
\]
Then $\e_k(A) = \d_1(A_k)$ for all $1 \leq k \leq n$, and the column class of $A$ is the product of the column classes of $A_1, \ldots, A_n$.
\end{lemma}
\begin{proof}
This result is obtained from the definitions of column class and elementary divisors by elementary considerations using the special structure of $A$.
\end{proof}

The next theorem is the main result of this section.

\begin{theorem}\label{texprodsuff}
Let $\a_1, \ldots, \a_n$ and $\b_1, \ldots, \b_n$ be ideals in $\o$ such that there exist matrices $A, B \in I_n$ with $\d_k(A) = \a_k$ and $\d_k(B) = \b_k$ for all $1 \leq k \leq n$. Then one can find matrices $\tilde{A}, \tilde{B} \in I_n$ satisfying $\d_k(\tilde{A}) = \a_k$ and $\d_k(\tilde{B}) = \b_k$ as well as $\d_k(\tilde{A} \tilde{B}) = \a_k \b_k$ for all $1 \leq k \leq n$.
\end{theorem}
\begin{proof}
Let 
\[
E' = \begin{pmatrix}
E_n \\
0
\end{pmatrix} \in \o^{2n \times n} \qquad\text{and}\qquad \tilde{E} = \begin{pmatrix}
E_n \quad 0
\end{pmatrix} \in \o^{n \times 2n}
\]
where $E_n$ denotes the ($n \times n$) identity matrix, and let $A' = E' A \tilde{E} = (\begin{smallmatrix} A & 0 \\ 0 & 0 \end{smallmatrix})$ as well as $B' = E' B \tilde{E} = (\begin{smallmatrix} B & 0 \\ 0 & 0 \end{smallmatrix})$. Denote by $A^*$ and $B^*$ the normal forms in the sense of \ref{tnormalform} of $A'$ and $B'$, respectively, and let $C^* = A^* (B^*)^\text{T}$ (where $(B^*)^\text{T}$ denotes the transpose of $B^*$). Furthermore, choose a matrix $C \in \o^{n \times n}$ satisfying $\e_k(C) = \e_k(A) \e_k(B)$ and thus $\d_k(C) = \d_k(A) \d_k(B)$ for all $1 \leq k \leq n$ according to \ref{texsingle}.

Let $C' = E' C \tilde{E} = (\begin{smallmatrix} C & 0 \\ 0 & 0 \end{smallmatrix})$ and show (using \ref{tequaldc}) that there exist $P, Q \in U_{2 n}$ satisfying $C' = P C^* Q$. By definition, $C^*$ is a block diagonal matrix, where the ($2 \times 2$) blocks $C_1, \ldots, C_n$ on the diagonal are of rank $1$ and satisfy $\d_1(C_k) = \e_k(A) \e_k(B)$ as well as $\mathfrak{C}(C_k) = \d_1(A_k) \mathfrak{H}$ for all $1 \leq k \leq n$, where $\mathfrak{H}$ denotes the principal ideal class of $K$. Since $\d_1(C_{k - 1}) = \e_{k - 1}(A) \e_{k - 1}(B) \divides \e_k(A) \e_k(B) = \d_1(C_k)$ holds for all $1 < k \leq n$, \ref{ldiag} implies $\e_k(C^*) = \e_k(A) \e_k(B) = \e_k(C) = \e_k(C')$ for all $1 \leq k \leq n$ and
\[
\mathfrak{C}(C^*) = \d_1(C_1) \cdots \d_1(C_n) \mathfrak{H} = \e_1(C) \cdots \e_n(C) \mathfrak{H} = \d_n(C) \mathfrak{H} = \mathfrak{H} \overset{C \in I_n}{=} \mathfrak{C}(C) = \mathfrak{C}(C').
\]
Since furthermore $\e_k(C^*) = \{0\} = \e_k(C')$ holds for all $n < k \leq 2 n$, \ref{tequaldc} implies the existence of $P, Q \in U_{2 n}$ such that $C' = P C^* Q$.

According to 2.3 there exist $P_1, Q_1, P_2, Q_2 \in U_{2 n}$ such that $A^* = P_1 A' Q_1$ and $B^* = P_2 B' Q_2$. Using $C = \tilde{E} C' E'$ as well as the definitions of $A'$ and $B'$, one obtains
\[
C = \underbrace{\tilde{E} P P_1 E'}_{=: R_1} A \underbrace{\tilde{E} Q_1 Q_2^\text{T} \tilde{E}^\text{T}}_{=: R_2} B^\text{T} \underbrace{(E')^\text{T} P_2^\text{T} Q E'}_{=: R_3}
\]
with $R_1, R_2, R_3 \in U_n$. Thus, $C \in U_n A U_n B^\text{T} U_n$ follows, and since $U_n B^\text{T} U_n = U_n B U_n$ (deducible from \ref{tequaldc}), one has $C \in U_n A U_n B U_n$ and thus can find $\tilde{A} \in U_n A U_n$ and $\tilde{B} \in B U_n$ satisfying $\tilde{A} \tilde{B} = C$. Since \cite{Steinitz:Rechteckige} 10 implies $\d_k(\tilde{A}) = \d_k(A)$ and $\d_k(\tilde{B}) = \d_k(B)$ for all $1 \leq k \leq n$, the already proven assertion $\d_k(C) = \d_k(A) \d_k(B)$ for all $1 \leq k \leq n$ yields the desired result.
\end{proof}

The translation of the last theorem into the language of realisability yields the following corollary.

\begin{corollary}
The triple $((\a_1, \ldots, \a_n), (\b_1, \ldots, \b_n), (\c_1, \ldots, \c_n))$ is realisable, if the conditions (1)--(3) of \ref{cexprodnec} are satisfied and $\c_k = \a_k \b_k$ holds for all $1 \leq k \leq n$.
\end{corollary}

To close this article, a characterisation of realisability in the special case that $\o$ is a principal ideal domain and $n = 2$ is investigated.

\begin{theorem}
Let $\o$ be a principle ideal domain. The triple $((\a_1, \a_2), (\b_1, \b_2), (\c_1, \c_2))$ is realisable, if and only if the following conditions are satisfied:
\begin{enumerate}[(1)]
\item $\a_1^2 \divides \a_2$ and $\b_1^2 \divides \b_2$ as well as $\c_1^2 \divides \c_2$ hold.
\item $\a_2 \b_2 = \c_2$.
\item $\a_1 \b_1 \divides \c_1 \divides \a_1 \b_1 (\smash{\a_1^{-2}} \a_2 + \smash{\b_1^{-2}} \b_2)$.
\end{enumerate}
\end{theorem}
\begin{proof}
It can be derived from \ref{cexprodnec} that the given conditions are necessary for realisability, so it remains to show that (1)--(3) imply realisability. Let (1)--(3) be satisfied. Denote by $a_1, a_2, b_1, b_2, c_1, c_2 \in \o$ generators of $\a_1, \a_2, \b_1, \b_2, \c_1, \c_2$, respectively. Let $d = c_1 a_1^{-1} b_1^{-1}$. Then $d \in \o$ according to (3). Furthermore, let
\[
A = \begin{pmatrix}
a_1 d & a_1 \\
a_2 a_1^{-1} & 0 \\
\end{pmatrix} \qquad\text{and}\qquad
B = \begin{pmatrix}
b_1 & 0 \\
0 & b_2 b_1^{-1}
\end{pmatrix}.
\]
Then $\d_1(A) = a_1 \o$, since $a_1 \divides a_2 a_1^{-1}$ according to (1); analogously, $\d_1(B) = b_1 \o$ holds. Moreover, $\d_2(A) = a_2 \o$ and $\d_2(B) = b_2 \o$ as well as $\d_2(A B) = c_2 \o$ are satisfied (the latter following from (2)), so it remains to prove $\d_1(A B) = c_1 \o$. For the first determinantal divisor one obtains
\[
\d_1(A B) = \d_1\left(\begin{pmatrix}
a_1 d b_1 & a_1 b_2 b_1^{-1} \\
a_2 a_1^{-1} b_1 & 0 \\
\end{pmatrix}\right) = a_1 b_1 (d \o + b_2 b_1^{-2} \o + a_2 a_1^{-2} \o),
\]
and since (3) implies $d = c_1 a_1^{-1} b_1^{-1} \divides a_2 a_1^{-2} \o + b_2 b_1^{-2} \o$, the above stated equation yields $\d_1(A B) = a_1 b_1 d \o = c_1 \o$, which completes the proof.
\end{proof}

\bibliography{detdiv}

\end{document}